\documentclass[11 pt]{amsart}
\usepackage{amsmath,enumerate,amsthm,amssymb,amscd}

\DeclareMathOperator{\Supp}{Supp}

\newcommand{\Spec}{\operatorname{Spec}}

\newcommand{\Hom}{\operatorname{Hom{}}}

\newcommand{\Tor}{\operatorname{Tor{}}}

\newcommand{\depth}{\operatorname{depth}}

\renewcommand{\hat}{\widehat}
\renewcommand{\phi}{\varphi}
\renewcommand{\to}{{\longrightarrow}}

\newcommand{\Ass}{\operatorname{Ass}}
\newcommand{\Min}{\operatorname{Min}}

\newcommand{\hgt}{\operatorname{ht}}

\newcommand{\Ann}{\operatorname{Ann}}
\newcommand{\ck}{\operatorname{v}}

\newtheorem{thm}{Theorem}[section]
\newtheorem{cor}[thm]{Corollary}
\newtheorem{prop}[thm]{Proposition}
\newtheorem{lemma}[thm]{Lemma}
\newtheorem{defn}[thm]{Definition}

\numberwithin{equation}{section}

\begin{document}
\title{The Frobenius functor and injective modules}
\author{Thomas Marley}
\address{Department of
Mathematics\\
University of Nebraska-Lincoln\\Lincoln,  NE 68588-0130}
\email{ tmarley@math.unl.edu}

\subjclass[2000]{ Primary
13H10; Secondary 13D45 }
\keywords{Frobenius map, injective module, canonical module}

\begin{abstract}
We investigate commutative Noetherian rings of prime characteristic such that the Frobenius functor applied to any injective module is again injective.  We characterize the class of one-dimensional local rings with this property and show that it includes all one-dimensional $F$-pure rings.  We also give a characterization of Gorenstein local rings in terms of $\Tor_i^R(R^{f},E)$, where $E$ is the injective hull of the residue field and $R^{f}$ is the ring $R$ whose right $R$-module action is given by the Frobenius map.
\end{abstract}

\date{February 28, 2012}

\bibliographystyle{amsplain}

\maketitle

\begin{section}{Introduction}
\end{section}

Let $R$ be a commutative Noetherian ring of prime characteristic $p$ and $f:R\to R$ the Frobenius ring
 homomorphism (i.e., $f(r)=r^p$ for $r\in R$).   We let $R^{f}$ denote the ring $R$ with the $R-R$
 bimodule structure given by $r\cdot s:=rs$ and $s\cdot r:=sf(r)$ for $r\in R$ and $s\in R^{f}$.
 Then $F_R(-):= R^{f}\otimes_R -$ is a right exact functor on the category of (left) $R$-modules and
 is called the {\it Frobenius functor} on $R$.    This functor has played an essential role in the solution of many important problems in commutative algebra for local rings of prime characteristic  (e.g., \cite{HR}, \cite{PS}, \cite{R}).  Of particular interest is how properties of the Frobenius map (or functor) characterize classical properties of the ring.   The most important result of this type, proved by Kunz \cite{K}, says that  $F_R$ is exact if and only if $R$ is a regular ring.   As another example, Iyengar and Sather-Wagstaff prove that a local ring $R$ is Gorenstein if and only if $R^{f}$ (viewed as a right $R$-module) has finite G-dimension \cite{IS}.

 As $F_R$ is additive and $F_R(R)\cong R$,
 it is easily seen that $F_R$ preserves projective (in fact, flat) modules.  In this paper, we consider rings for
 which $F_R$ preserves injective modules, i.e., rings $R$ having the property that $F_R(I)$ is injective for
 every injective $R$-module $I$.
 It is well known (e.g., \cite{HS}) that Gorenstein rings have this property, and
 in fact this is true for quasi-Gorenstein rings as well (Proposition \ref{quasi}).  In Section 3,
 we show that if $F_R$ preserves injectives then $R$ satisfies Serre's condition $S_1$ and that
 $F_R(I)\cong I$ for every injective $R$-module $I$.   Moreover, if
 $R$ is the homomorphic image of a Gorenstein local ring, then $F_R$
 preserves all injectives if and only if $F_R(E)$ is injective,
 where $E$ is the injective hull of the residue field.
 We also give a criterion (Theorem \ref{Tor-condition}) for a local ring $R$ to be Gorenstein in terms of
 $\Tor_i^R(R^f,E)$:

\begin{thm}  Let $(R,m)$ be a local ring and $E=E_R(R/m)$.
Then the following are equivalent:
\begin{enumerate}[(a)]
\item $R$ is Gorenstein;
\item $\Tor^R_0(R^f,E)\cong E$ and $\Tor_i^R(R^{f},E)=0$ for $i=1,\dots,\depth R$.
\end{enumerate}
\end{thm}

In Section 4, we study one-dimensional rings $R$ such that $F_R$ preserves injectives.  In particular, we give the following characterization (Theorem \ref{one-dim}) in the case $R$ is local:

\begin{thm}  Let $(R,m)$ be a one-dimensional local ring and $E=E_R(R/m)$.  The following conditions are equivalent:
\begin{enumerate}[(a)]
\item $F_R(E)$ is injective;
\item $F_R(I)\cong I$ for all injective $R$-modules $I$;
\item $R$ is Cohen-Macaulay and has a canonical ideal $\omega_R$ such that $\omega_R\cong \omega_R^{[p]}$.
\end{enumerate}
\end{thm}

\noindent Using this characterization, we show that every
one-dimensional $F$-pure ring preserves injectives. We also prove,
using a result of Goto \cite{G}, that if $R$ is a complete
one-dimensional local ring with algebraically closed residue field
and has at most two associated primes then $R$ is Gorenstein if and
only if $F_R$ preserves injectives. We remark that Theorems 1.1 and
1.2 are dual to results appearing in \cite{G} in the case the
Frobenius map is finite.  This duality is made explicit in
Proposition \ref{prop2}.

In Section 2, we summarize several results concerning the Frobenius functor and canonical modules which will be needed in the later sections.  Most of these are well-known, but for some we could not find a reference in the literature.

\medskip
\noindent {\bf Acknowledgment:}  The author would like to thank Sri
Iyengar for showing him a proof of Corollary \ref{cor1}.  He also
grateful to Neil Epstein, Louiza Fouli, Craig Huneke,  and Claudia
Miller for helpful conversations concerning this material.

\begin{section}{Some properties of the Frobenius functor and canonical modules}
\end{section}

Throughout this paper $R$ denotes a commutative Noetherian
ring of prime characteristic $p$.     For an $R$-module $M$, $E_R(M)$ will denote the injective hull of $M$.  If $I$ is an ideal of $R$ then $H^i_I(M)$ will denote the $i$th local cohomology module of $M$ with support in $I$.   If $R$ is local with maximal ideal $m$, we denote the $m$-adic completion of $R$ by $\hat R$. We refer the reader to \cite{BH} or \cite{Mat} for any unexplained terminology or notation.

Let $M$ be a finitely generated $R$-module with presentation
$R^r\xrightarrow{\phi} R^s\to M\to 0$, where $\phi$ is represented (after fixing bases) by an $s\times r$ matrix $A$.   Then $F_R(M)$ has the presentation $R^r\xrightarrow{F_R(\phi)}R^s\to F_R(M)\to 0$
and the map $F_R(\phi)$ is represented by the matrix $A^{[p]}$ obtained by raising the corresponding entries of $A$ to the $p$th power.  For an ideal $I$ of $R$ and $q=p^e$, we let $I^{[q]}$ denote the ideal generated by the set $\{i^q\mid i\in I\}$.   Note that, by above,  $F^e_R(R/I)\cong R/I^{[q]}$, where $F^e_R$ is the functor $F_R$ iterated $e$ times.

The following proposition lists  a few properties of the Frobenius functor which we will use in the sequel.
Most of these are well known:

\begin{prop} \label{frob} The following hold for any $R$-module $M$:
\begin{enumerate}[(a)]
\item If $T$ is an $R$-algebra then $F_T(T\otimes_RM)\cong T\otimes_R F_R(M)$.
\item If $S$ is a multiplicatively closed set of $R$ then $F_R(M_S)\cong F_{R_S}(M_S)\cong R_S\otimes_R F_R(M)$.
\item $\Supp_R F_R(M)=\Supp_R M$.
\item $M$ is Artinian if and only if $F_R(M)$ is Artinian.
\item If $(R,m)$ is local, and $M$ is finitely generated of dimension $s$, then
$F_R(H^s_m(M))\cong H^s_m(F_R(M))$.
\end{enumerate}
\end{prop}
\begin{proof} The first two parts follow easily from properties of tensor products.   For (c),
it suffices to show $\Supp_R M\subseteq \Supp_R F_R(M)$.    Using
(a) and (b), it is enough to prove that if $R$ is a complete local
ring and $M\neq 0$ then $F_R(M)\neq 0$. Then $R=S/I$ where $S$ is a
regular local ring of characteristic $p$. Let $Q\in \Ass_S M$.  Then
there is an exact sequence $0\to S/Q\to M$.  As $S$ is regular,
$F_S$ is exact and we have an injection  $S/Q^{[p]}\to F_S(M)$.
Hence, $F_S(M)\neq 0$.   By part (b) $F_R(M)\cong
S/I\otimes_S F_S(M)$.   As $IM=0$, $I^{[p]}F_S(M)=0$.  Let $t$ be an
integer such that $I^t\subseteq I^{[p]}$.   Now, if $F_R(M)=0$ then
$F_S(M)=IF_S(M)$.  Iterating, we have $F_S(M)=I^tF_S(M)=0$, a
contradiction.  Hence, $F_R(M)\neq 0$.

For (d),  since $\Supp_RM=\Supp_R F_R(M)$ and the support of an Artinian module is finite, it suffices to consider the case $(R,m)$ is a local ring.   Furthemore, for any $R$-module $N$ with $\Supp_R N\subseteq \{m\}$, we have $\hat R\otimes_R N\cong N$.  Hence, we may assume $R$ is complete.   Then $R=S/I$ where $S$ is a regular local ring of characteristic $p$.   As $F_R(M)\cong S/I\otimes_S F_S(M)$, it is clear that if $F_S(M)$ is Artinian then so is $F_R(M)$.  Conversely, if $S/I\otimes_SF_S(M)$ is Artinian, then $F_S(M)$ is Artinian since we also have $I^{[p]}F_S(M)=0$.   Thus, it is enough to prove the result in the case that $R$ is a regular local ring.   Recall that an $R$-module is Artinian if and only if $\Supp_RM\subseteq \{m\}$ and $(0:_Mm)$ is finitely generated.  Since $\Supp_R M=\Supp_R F_R(M)$, it suffices to prove that $(0:_Mm)$ is finitely generated if and only if $(0:_{F_R(M)}m)$ is finitely generated.
As $F_R$ is exact, $F_R((0:_Mm))\cong (0:_{F_R(M)}m^{[p]})$.  But $(0:_Mm)$ is finitely generated if and only if $F_R((0:_Mm))$ is finitely generated, and $(0:_{F_R(M)}m^{[p]})$ is finitely generated if and only if $(0:_{F_R(M)}m)$ is finitely generated.

For (e), let $I=\Ann_RM$ and choose $x_1\dots,x_s\in m$ such that
their images in $R/I$ form a system of parameters. Set
$J=(x_1,\dots,x_s)$.  Then $H^s_m(M)\cong H^s_{J}(M)$. Since
$H^i_{J}(R)=0$ for all $i>s$, $T\otimes_R H^s_{J}(M)\cong
H^i_{TJ}(T\otimes_R M)$ for any $R$-algebra $T$.  Then
$$F_R(H^s_m(M))\cong R^{f}\otimes_R H^s_J(M)\cong
H^s_{J^{[p]}}(F_R(M)).$$ Finally,
 as $I^{[p]}\subseteq \Ann_RF_R(M)$ and $J^{[p]}+I^{[p]}$ is
$m$-primary, we have $H^s_{J^{[p]}}(F_R(M))\cong H^s_m(F_R(M))$.
\end{proof}

We need one more result concerning the Frobenius, which is again well known:

\begin{lemma} \label{tor} Let $(R,m)$ be a local ring of dimension $d$.  If $R$ is Cohen-Macaulay then $\Tor_i^R(R^{f},H^d_m(R))=0$ for all $i\ge 1$.
\end{lemma}
\begin{proof} Let $\bold x=x_1,\dots,x_d$ be a system of parameters for $R$ and $C(\bold x)$
the \v Cech cochain complex with respect to $\bold x$.  Note that $F_R(C(\bold x))\cong C(\bold x^p)$ where $\bold x^p=x_1^p,\dots,x_d^p$.  Since $R$ is Cohen-Macaulay,  $\bold x$ is a regular sequence and thus $C(\bold x)$ is a flat resolution of $H^d_m(R)$.  Hence for $i\ge 1$,
$$\Tor_i^R(R^{f},H^d_m(R))\cong \operatorname{H}^{d-i}(R^f\otimes_R C(\bold x))\cong \operatorname{H}^{d-i}(C(\bold x^p))=0.$$
\end{proof}

For a nonzero finitely generated $R$-module $M$ we let $U_R(M)$ be the intersection of all
the primary components $Q$ of $0$ in $M$ such that $\dim M/Q=\dim
M$.  It is easily seen that $U_R(M)=\{x\in M\mid \dim Rx<\dim M\}$.
A local ring  $R$ is said to be {\it unmixed} if $U_{\hat R}(\hat R)=0$.

Let $(R,m)$ be a local ring of dimension $d$, $E=E_R(R/m)$, and
$(-)^{\ck}:=\Hom_R(-,E)$ the Matlis dual functor. A finitely
generated $R$-module $K$ is called a {\it canonical module} of $R$
if $K\otimes_R \hat R\cong H^d_m(R)^{\ck}$.   If a canonical module exists, it is
unique up to isomorphism and denoted by $\omega_R$.  Any complete local ring possesses a canonical module.   More generally, $R$ possesses a canonical module if  $R$ is the
homomorphic image of a Gorenstein ring.  Proofs of these facts can be found in \cite{A} (or
the references cited there). We summarize some additional properties
of canonical modules in the following proposition:

\begin{prop}  \label{canon} Let $R$ be a local ring which possesses a canonical module $\omega_R$ and let \newline $h: R\to \Hom_R(\omega_R,\omega_R)$ be the
natural map.  The following hold:
\begin{enumerate}[(a)]
\item $\Ann_R\omega_R=U_R(R)$.
\item $(\omega_R)_P\cong \omega_{R_P}$ for every prime $P\in \Supp_R \omega_R$.
\item  $\omega_R\otimes_R \hat R\cong \omega_{\hat R}$.
\item $\ker h=U_R(R)$.
\item $h$ is an isomorphism if and only if $R$ satisfies
Serre's condition $S_2$.
\item If $R$ is complete, $\Hom_R(M,\omega_R)\cong H^d_m(M)^{\ck}$ for any $R$-module $M$.
\end{enumerate}
\end{prop}
\begin{proof}  The proofs of parts (a)-(e) can be found in \cite{A}.  Part (f) is just  local duality, but it can also be seen directly from the definition of $\omega_R$ and adjointness:  $$\Hom_R(H^d_m(M), E)\cong \Hom_R(M\otimes_RH^d_m(R), E)\cong  \Hom_R(M, H^d_m(R)^{\ck}).$$
\end{proof}

If $R$ is a local ring possessing a canonical module $\omega_R$ such
that $\omega_R\cong R$, then $R$ is said to be {\it
quasi-Gorenstein}.  Equivalently, $R$ is quasi-Gorenstein if and
only if $H^d_m(R)\cong E$.  By the proposition above, if $R$ is
quasi-Gorenstein then $R$ is $S_2$ and $R_P$ is quasi-Gorenstein for
every $P\in \Spec R$.  It is easily seen that $R$ is Gorenstein if
and only if $R$ is Cohen-Macaulay and quasi-Gorenstein.   Finally, there exist quasi-Gorenstein rings which are not Cohen-Macaulay (e.g., \cite{A}).

\begin{section}{Rings for which Frobenius preserves injectives}
\end{section}

To facilitate our discussion we make the following definition:

\begin{defn}{\rm A Noetherian ring of characteristic $p$ is said to be {\it FPI} (i.e., `Frobenius
Preserves Injectives') if $F_R(I)$ is injective for every
injective $R$-module $I$.  We say that $R$ is {\it weakly FPI} if
$F_R(I)$ is injective for every Artinian injective $R$-module $I$.}
\end{defn}

We note that $R$ is FPI (respectively, weakly FPI) if and only if
$F_R(E_R(R/P))$ is injective for every prime (respectively, maximal)
ideal $P$ of $R$.   Also, since Frobenius commutes with
localization, $R$ is FPI if and only if $R_P$ is weakly FPI for
every prime ideal $P$.

\begin{prop} \label{frob-injective} Let $I$ be an injective $R$-module and suppose $F_R(I)$ is injective.  Then $F_R(I)\cong I$.
\end{prop}
\begin{proof}  By the remarks above, it suffices to prove this in the case $R$ is local and $I=E_R(R/m)$ where $m$ is the maximal ideal.    And as $E_R(R/m)\cong E_{\hat R}(\hat R/\hat m)$, we may also assume $R$ is complete.  Let $E=E_R(R/m)$ and $d=\dim R$.  Since $E$ is Artinian, $F_R(E)$ is Artinian by Proposition \ref{frob}(d).  Hence, $F_R(E)\cong E^n$ for some integer $n\ge 1$.  It suffices to show that $n=1$.  Let $U=U_R(R)$.  By  parts (d) and (f) of Proposition
\ref{canon}, we have an exact sequence $0\to R/U\to H^d_m(\omega_R)^{\ck}$.  Dualizing, we obtain a surjection $H^d_m(\omega_R)\to E_{R/U}\to 0$, where $E_{R/U}:=E_{R/U}(R/m)\cong \Hom_R(R/U,E)$.   Since $\omega_R$ is a finitely generated $R$-module, we have a surjection $R^s\to \omega_R\to 0$ for some $s$.  This yields an exact sequence $H^d_m(R)^s\to H^d_m(\omega_R)\to 0$.   Composing, we obtain an exact sequence
\begin{align*}H^d_m(R)^s\to E_{R/U}\to 0.\tag{*}
\end{align*}
Now consider the exact sequence $0\to U\to R\to R/U\to 0$.  Applying the Matlis dual, we have that
$$0\to E_{R/U}\to E\to U^{\ck}\to 0$$ is exact.  Combining with  (*), we obtain an exact sequence
$$H^d_m(R)^s\to E\to U^{\ck}\to 0.$$
Applying $F_R^e$ and using that $F_R(E)\cong E^n$, we obtain the exact sequence
$$H^d_m(R)^s\to E^{n^e}\to F_R^e(U^{\ck})\to 0.$$ Dualizing again, we have that
\begin{align*}
0\to F_R^e(U^{\ck})^{\ck}\to R^{n^e}\to \omega_R^s\tag{**}
\end{align*}
is exact.  Let $I=\Ann_R U=\Ann_R U^{\ck}$.   Since $\dim R/I<d$ and $I^{[q]}\subseteq \Ann_R F_R^e(U^{\ck})=\Ann_R F_R^e(U^{\ck})^{\ck}$ (where $q=p^e$), we have $\dim F^e(U^{\ck})^{\ck}<d$.  Let $P$ be a prime ideal of $R$ such that $\dim R/P=d$.  Localizing (**) at $P$, we obtain an injection $0\to R_P^{n^e}\to \omega_{R_P}^s$.  If $n>1$, this is easily seen to be a contradiction by comparing lengths.
\end{proof}

We summarize some properties of FPI rings in the following proposition.
Recall that a ring $R$ is said to be {\it generically Gorenstein} if $R_P$ is Gorenstein for every $P\in \Min _RR$.

\begin{prop} \label{fpi}  Let $R$ be a Noetherian ring.
\begin{enumerate}[(a)]
\item If $R$ is FPI and $S$ is a multiplicatively closed set of $R$ then $R_S$ is FPI.
\item If $R$ is FPI then $R$ is generically Gorenstein.
\item If $R$ is local then $R$ is weakly FPI if and only if $\hat R$ is weakly FPI.
\item Let $S$ be a faithfully flat $R$-algebra which is FPI and suppose  that the fibers $k(P)\otimes_RS$ are generically Gorenstein for all $P\in \Spec R$.   Then $R$ is FPI.
\item Suppose $R$ is the homomorphic image of a Gorenstein local ring.  If $\hat R$ is FPI then so is $R$.
\end{enumerate}
\end{prop}
\begin{proof}  Part (a) follows easily from the fact that Frobenius commutes with localization. To prove (b),  it suffices to prove that if $R$ is local, zero-dimensional, and FPI then $R$ is Gorenstein.  In this situation, note that if $M$ is a finitely generated $R$-module then $F^e_R(M)$ is free for sufficiently large $e$.  Hence, as $F_R(E)\cong E$, $E$ must be free.  This implies $R$ is injective and therefore Gorenstein.

Part (c) follows from the fact that $E_R(R/m)\cong E_{\hat R}(\hat R/\hat m)$.

To prove (d), let $P\in \Spec R$ and $E=E_R(R/P)$.  It suffices to
show that $F_R(E)$ is injective. Let $Q\in \Spec S$ which is minimal
over $PS$.  Then $S_Q$ is a faithfully flat $R_P$-algebra and is FPI
by part (a).  Hence, we may assume $(R,m)$ and $(S,n)$ are local,
$E=E_R(R/m)$, and the fiber $S/mS$ is zero-dimensional Gorenstein.
By \cite[Theorem 1]{F}, $S\otimes_RE$ is injective. As $S$ is FPI,
$S\otimes_R F_R(E)\cong F_S(S\otimes_RE)$ is injective.  Since $S$
is faithfully flat over $R$, this implies $F_R(E)$ is injective.

Part (e) follows from (d) since the hypothesis implies that the formal fibers of $R$ are Gorenstein.
\end{proof}

The following result is essentially \cite[Proposition 1.5]{HS}:

\begin{prop} \label{quasi} Let $R$ be a quasi-Gorenstein local ring.  Then $R$ is FPI.
\end{prop}
\begin{proof}  Let $P\in \Spec R$ and $E=E_R(R/P)$.  It suffices to show that $F_R(E)$ is injective.  Since $R_P$ is quasi-Gorenstein, we may assume $P=m$.   Then $E\cong H^d_m(R)$ where $d=\dim R$.  Hence,  by Proposition \ref{frob}(e), $F_R(E)\cong F_R(H^d_m(R))\cong H^d_m(R)\cong E$.
\end{proof}

Next, we show that for a large class of rings, weakly FPI implies
FPI.  First, we prove a few preliminary results.

\begin{lemma} \label{lem1} Let $\phi:R\to S$ be a homomorphism of
local rings such that $S$ is finite as an $R$-module.  Let $k$ and
$\ell$ denote the residue fields of $R$ and $S$, respectively, and
set $E_R=E_R(k)$ and $E_S=E_S(\ell)$.  Then $\Hom_R(S,E_R)\cong
E_S$. \end{lemma}
\begin{proof} Clearly, $\Hom_R(S,E_R)$ is an Artinian injective
$S$-module.  Hence, $\Hom_R(S,E_R)\cong E_S^n$ for some $n$.  It
suffices to show that $n=1$.  Since $S$ is a finite $R$-module,
$\ell=k^t$ for some $t$.  Then $nt=\dim_k\Hom_S(\ell, E_S^n)$. But
$\Hom_S(\ell, E_S^n)\cong \Hom_S(\ell, \Hom_R(S,E_R))\cong
\Hom_R(k^t, E_R)\cong k^t$ as $R$-modules.  Thus, $n=1$.
\end{proof}

\begin{prop} \label{prop1} Let $\phi:R\to S$ be a homomorphism of local rings such
that $S$ is a finite $R$-module.  Let $E_R$ and $E_S$ be as in Lemma
\ref{lem1}.  Then for any $R$-module $M$ we have
$$\Hom_S(S\otimes_RM, E_S)\cong \Hom_R(S, \Hom_R(M,E_R)).$$
\end{prop}
\begin{proof} Using adjointness and Lemma \ref{lem1}, we have
\begin{align*}
\Hom_S(S\otimes_RM, E_S)&\cong \Hom_S(S\otimes_RM, \Hom_R(S,E_R))\\
&\cong \Hom_R(S\otimes_RM, E_R)\\
&\cong \Hom_R(S, \Hom_R(M,E_R)).
\end{align*}
\end{proof}

Recall that a ring $R$ of characteristic $p$ is called {\it
$F$-finite} if the Frobenius map $f:R\to R$ is a finite morphism;
i.e., $R^f$ is finite as a right $R$-module. Also, for a right
$R$-module $M$ let $\Hom_R(R^{f},M)$ denote the set of right
$R$-module homomorphisms from $R^{f}$ to $M$ and viewed as a left
$R$-module in the natural way.

\begin{cor} \label{cor1} Let $R$ be an $F$-finite local ring. Then
$F_R(M)^{\ck}\cong \Hom_R(R^f,M^{\ck})$ for any $R$-module $M$.
\end{cor}
\begin{proof} This is just a restatement of Proposition \ref{prop1},
where $S$ is the ring $R$ viewed as an $R$-module via $f$.
\end{proof}

\begin{prop} \label{prop2} Suppose $R$ is an $F$-finite local ring
and $E$ the injective hull of the residue field of $R$.  The
following are equivalent:
\begin{enumerate}
\item $\Hom_R(R^f,R)\cong R$;
\item $F_R(E)\cong E$.
\end{enumerate}
\end{prop}
\begin{proof} Without loss of generality, we may assume $R$ is
complete.  Then by Corollary \ref{cor1}, we have $F_R(E)^{\ck}\cong
\Hom_R(R^f,R)$.  The result now follows.
\end{proof}

\begin{thm}\label{fpi=wfpi} Let $R$ be the homomorphic image of a Gorenstein ring.
Then $R$ is $FPI$ if and only if $R$ is weakly FPI.
\end{thm}
\begin{proof} It suffices to prove this in the case $R$ is a local
ring with maximal ideal $m$.  We first prove the theorem in the case
$R$ is $F$-finite.  Suppose $F_R(E)\cong E$ and let $P\in \Spec R$.
By Proposition \ref{prop2}, we have $\Hom_R(R^f,R)\cong R$.
Localizing, we have $\Hom_{R_P}(R_P^f,R_P)\cong R_P$.  Since $R_P$
is $F$-finite, we have again by Proposition \ref{prop1} that
$F_{R_P}(E_{R_P}(k(P))\cong E_{R_P}(k(P))$; i.e., $R_P$ is weakly
FPI.  As $P$ was arbitrary, this shows that $R$ is FPI.

For the general case, by parts (c) and (d) of Proposition \ref{fpi},
we may assume $R$ is complete.  Let $k$ be the residue field of $R$.
By the Cohen Structure Theorem, $R\cong A/I$ where
$A=k[[T_1,\dots,T_n]]$ and $T_1,\dots, T_n$ are indeterminates.  Let
$\ell$ be the algebraic closure $k$, $B=\ell[[T_1,\dots,T_n]]$, and
$S=B/IB$.   Note that as $B$ is faithfully flat over $A$, $S$ is
faithfully flat over $R$, and that $S$ is $F$-finite. Also, $S$ is
weakly FPI, as $E_S(\ell)\cong E_R(k)\otimes_RS$ (\cite[Theorem
1]{F}). Thus, $S$ is FPI by the $F$-finite case. As the fibers of
$S$ over $R$ are Gorenstein, we have that $R$ is FPI by Proposition
\ref{fpi}(d).
\end{proof}

We next show that a weakly FPI ring has no embedded associated primes:

\begin{prop} \label{S_1} Let $R$ be a weakly FPI ring.  Then $R$ satisfies Serre's condition $S_1$.
\end{prop}
\begin{proof}
Without loss of generality, we may assume that $R$ is local and complete.  Let $P\in \Spec R$ and $s=\dim R/P$.     Since $\omega_{R/P}$ is a rank one torsion-free $R/P$-module, there exists an exact sequence $0\to \omega_{R/P}\to R/P$.   By Matlis duality, we have the exact sequence
$$E_{R/P}\to H^s_m(R/P)\to 0,$$   where $E_{R/P}:=E_{R/P}(R/m)$.  Applying $F^e_R$ to this sequence, we obtain
$$F^e_R(E_{R/P})\to H^s_m(R/P^{[q]})\to 0$$ is exact, where $q=p^e$.   By Proposition \ref{canon}(a), $\Ann_R H^s_m(R/P^{[q]})=U_R(R/P^{[q]})$.
Note that as $\Min_R R/P^{[q]}=\{P\}$, $U_R(R/P^{[q]})=\psi^{-1}(P^{[q]}R_P)$, where $\psi:R\to R_P$ is the natural map.  Hence, for all $q=p^e$ we have
\begin{align*}
\Ann_R F^e_R(E_{R/P})\subseteq \psi^{-1}(P^{[q]}R_P). \tag{\#}
\end{align*}

Now suppose that $P\in \Ass_R R$.  Then there exists an exact sequence $0\to R/P\to R$.  Dualizing, we have $E\to E_{R/P}\to 0$ is exact, where $E=E_R(R/m)$.
Applying $F^e_R$ and using that $F_R(E)\cong E$, we have the exact sequence
$$E\to F^e_R(E_{R/P})\to 0.$$
Dualizing again, we have that
$$0\to F^e_R(E_{R/P})^{\ck}\to R$$
is exact.  Note that as $PE_{R/P}=0$, $P^{[q]}F^e_R(E_{R/P})=P^{[q]} F^e_{R}(E_{R/P})^{\ck}=0$ for all $q=p^e$.  Hence, for all $q$ we have an exact sequence
$$0\to F^e_R(E_{R/P})^{\ck}\to H^0_P(R).$$  Therefore, there exists a positive integer $n$ such that $P^n \subseteq \Ann_R F^e_R(E_{R/P})$ for all $e$.  By (\#),  this implies $P^nR_P\subseteq P^{[q]}R_P$
for all $q=p^e$.  Hence, $P^nR_P=0$ and $\hgt P=0$.
\end{proof}

In general, if $R$ is  weakly FPI and $x$ is a non-zero-divisor on
$R$ then $R/(x)$ need not be weakly  FPI.   (Otherwise, using
Propositions \ref{fpi}(b) and \ref{S_1}, one could prove that every
weakly FPI ring is Gorenstein; but there exist quasi-Gorenstein
rings which are not Gorenstein.) However, this does hold if in
addition we have $\Tor^R_1(R^{f},E)=0$.  We'll use this to give a
criterion for $R$ to be Gorenstein in terms of $\Tor^R_i(R^f, E)$.
First, we prove the following lemma:

\begin{lemma} \label{regular} Let $(R,m)$ be a local ring and $I$  an ideal generated by a regular sequence.  Then
$$R/I\otimes_R E_{R/I^{[p]}}(R/m)\cong E_{R/I}(R/m).$$
\end{lemma}
\begin{proof}  Without loss of generality, we may assume $R$ is complete. Let $E=E_R(R/m)$. Note that $E_{R/I^{[p]}}(R/m)\cong \Hom_R(R/I^{[p]}, E)$ and $E_{R/I}(R/m)\cong \Hom_R(R/I,E)$.  Taking Matlis duals it suffices to prove that $\Hom_R(R/I, R/I^{[p]})\cong R/I$.  But this is easily seen to hold as $I$ is generated by a regular sequence.\end{proof}

The following result is dual to  Theorem 1.1 of \cite{G}, which
holds in the case the Frobenius map is a finite morphism.

\begin{thm} \label{Tor-condition} Let $(R,m)$ be a local ring and $E=E_R(R/m)$.   The following conditions are equivalent:
\begin{enumerate}
\item $F_R(E)\cong E$ and $\Tor_i^R(R^{f},E)=0$ for all $i=1,\dots, \depth R$;
\item $R$ is Gorenstein.
\end{enumerate}
\end{thm}
\begin{proof} Condition (2) implies (1) by Lemma \ref{tor} and Proposition \ref{quasi} (note $E\cong H^d_m(R)$).    Conversely,  suppose condition (1) holds.  Let  $\bold x= x_1,\dots,x_r\in m$ be a maximal regular sequence on $R$ and $K(\bold x)$ the Koszul complex with respect to $\bold x$.  Then $K(\bold x)\xrightarrow{\epsilon} R/(\bold x)\to 0$ is exact, where $\epsilon$ is the augmentation map.   Dualizing, we have $0\to E_{R/(\bold x)}(R/m) \to K(\bold x)^{\ck}$ is exact.  Since $K(\bold x)_j^{\ck}\cong E^{\binom{r}{j}}$ for all $j$ and $\Tor^R_i(R^{f}, E)=0$ for $1\le i\le r$, we obtain that $0\to F_R(E_{R/(\bold x)}(R/m))\to F_R(K(\bold x)^{\ck})$ is exact.  In particular, since $F_R(E)\cong E$, we have that
$$0\to F_R(E_{R/(\bold x)}(R/m))\to E \xrightarrow{[x_1^p \dots x_r^p]} E^r$$
is exact.   Hence,
$$F_R(E_{R/(\bold x)}(R/m))\cong \Hom_R(R/(\bold x)^{[p]}, E)\cong E_{R/(\bold x)^{[p]}}(R/m).$$
Using Lemma \ref{regular}, we have
\begin{align*}
F_{R/(\bold x)}(E_{R/(\bold x)}(R/m))&\cong R/(\bold x)\otimes_R F_R(E_{R/(\bold x)}(R/m))\\
&\cong R/(\bold x)\otimes_R E_{R/(\bold x)^{[p]}}(R/m)\\
&\cong E_{R/(\bold x)}(R/m).
\end{align*}
This says that $R/(\bold x)$ is weakly FPI.  Since $\depth R/(\bold x)=0$, we must have $\dim R/(\bold x)=0$ by Proposition \ref{S_1}.  But then $R/(\bold x)$ is Gorenstein by Proposition \ref{fpi}(b).  Hence, $R$ is Gorenstein.
\end{proof}

\begin{section}{One-dimensional FPI rings}
\end{section}

We now turn our attention to the one-dimensional case.    If $R$ is
a local ring possessing an ideal which is also a canonical module of
$R$, this ideal is referred to as a  {\it canonical ideal} of $R$.
If $(R,m)$ is a one-dimensional Cohen-Macaulay local ring, then $R$
has a canonical ideal (necessarily $m$-primary) if and only if $\hat
R$ is generically Gorenstein (\cite[Satz 6.21]{H2}). The following
can viewed as a generalization of Lemma 2.6 of \cite{G}, which holds
in the case Frobenius map is finite:

\begin{thm} \label{one-dim} Let $(R,m)$ be a one-dimensional local ring.  The following conditions are equivalent:
\begin{enumerate}[(a)]
\item $R$ is weakly FPI;
\item $R$ is FPI;
\item $R$ is Cohen-Macaulay and has a canonical ideal $\omega_R$ such that $\omega_R\cong \omega_R^{[p]}$.
\end{enumerate}
\end{thm}
\begin{proof}
Since (b) trivially implies (a), it suffices to prove (a) implies (c) and (c) implies (b).

We first prove (a) implies (c):  As $R$ is weakly FPI, $R$ is
Cohen-Macaulay by  Proposition \ref{S_1}.   Furthemore, $\hat R$ is
weakly FPI and thus FPI by Theorem \ref{fpi=wfpi}.  Thus, $\hat R$
is generically Gorenstein, which implies $R$ possesses a canonical
ideal $\omega_R$.  To show $\omega_R\cong \omega_R^{[p]}$, it
suffices to show that $\omega_R^{[p]}$ is a canonical ideal of $R$.
Since $\omega_{R}^{[p]}$ is a canonical ideal for $R$ if and only if
$\omega_R^{[p]}\otimes_R\hat R\cong (\omega_R\hat R)^{[p]}$ is a
canonical ideal for $\hat R$, we may assume without loss of
generality that $R$ is complete.  Since $R$ is Cohen-Macaulay,
$H^1_m(\omega_R)\cong E$, where $E=E_R(R/m)$. Applying local
cohomology to the exact sequence
$$0\to \omega_R \to R\to R/\omega_R\to 0$$
we obtain
$$0\to R/\omega_R\to E\to H^1_m(R)\to 0$$
is exact. Applying $F_R$ and using that $\Tor_1^R(R^{f},H^1_m(R))=0$
(Lemma \ref{tor}), we have the exact sequence
$$0\to R/\omega_R^{[p]}\to E\to H^1_m(R)\to 0.$$
Dualizing, we have
$$0\to \omega_R \to R \to \Hom_R(R/\omega_R^{[p]}, E)\to 0$$
is exact.
But as $0\to \Hom_R(R/m,R/\omega_R^{[p]})\to \Hom_R(R/m,E)$
is exact, the socle of $R/\omega_R^{[p]}$ is one-dimensional and hence $R/\omega_R^{[p]}$ is Gorenstein.  Thus, $\Hom_R(R/\omega_R^{[p]},E)\cong R/\omega_R^{[p]}$ and we obtain an exact sequence
$$0\to \omega_R \to R \to R/\omega_R^{[p]}\to 0.$$  This implies that $\omega_R\cong \omega_R^{[p]}$.

Next, we prove (c) implies (b):   Since $R$ is Cohen-Macaulay and
possesses a canonical ideal, $R$ is the homomorphic image of a
Gorenstein ring (cf. \cite[Theorem 3.3.6]{BH}).  Hence, by Theorem
\ref{fpi=wfpi}, it suffices to prove that $R$ is weakly FPI.    Let
$\pi: F_R(\omega_R)\to \omega_R^{[p]}$ be the natural surjection
given by $\pi(r\otimes u)=ru^p$ and let $C=\ker \pi$.  Since $R_P$
is Gorenstein for all primes $P\neq m$, $\dim C=0$. Consequently,
$H^1_m(C)=0$ and  $H^1_m(F_R(\omega_R))\cong H^1_m(\omega_R^{[p]})$.
Since $E\cong H^1_m(\omega_R)$ and $\omega_R\cong \omega_R^{[p]}$,
we have
$$F_R(E)\cong F_R (H^1_m(\omega_R))\cong H^1_m(F_R(\omega_R))\cong H^1_m(\omega_R^{[p]})\cong H^1_m(\omega_R)\cong E.$$
Hence, $R$ is weakly FPI.
\end{proof}

We remark that there exists one-dimensional local FPI rings which are not Gorenstein.  In fact, the next result shows that every one-dimensional $F$-pure ring is FPI.    Recall that a homomorphism $A\to B$ of commutative rings is called {\it pure} if the map $M\to B\otimes_AM$ is injective for every $A$-module $M$.  A ring $R$ of prime characteristic  is called {\it F-pure} if the Frobenius map $f:R\to R$ is pure.

\begin{prop} \label{Fpure} Let $(R,m)$ be a one-dimensional $F$-pure ring.  Then $R$ is FPI.
\end{prop}
\begin{proof}  We may assume $R$ is local.  By Theorem \ref{one-dim},
it suffices to show that $R$ is weakly FPI.  By Proposition
\ref{fpi}(c) and \cite[Corollary 6.13]{HR}, we may assume $R$ is
complete.   Let $k$ be the residue field of $R$.   By the Cohen
Structure Theorem, $R\cong A/I$ where $A=k[[T_1,\dots,T_n]]$ and
$T_1,\dots, T_n$ are indeterminates.  Let $\ell$ be the algebraic closure
$k$, $B=\ell[[T_1,\dots,T_n]]$, and $S=B/IB$.   Note that as $B$ is
faithfully flat over $A$, $S$ is faithfully flat over $R$.
Furthermore, $S$ is $F$-pure since $R$ is (e.g., \cite[Theorem
1.12]{Fe}).  Finally, by \cite[Theorem 1]{F}, $E_S(\ell)\cong
E_R(k)\otimes_RS$.  Hence, $S$ is weakly FPI if and only if $R$ is
weakly FPI. Thus, resetting notation, we may assume $R$ is complete
and its residue field $k$ is algebraically closed. By \cite[Theorem
1.1]{GW}, $R\cong k[[T_1,\dots,T_n]]/I$ where $I=(\{T_iT_j\mid 1\le
i<j\le n\})$. It is easily checked that
$\omega_R=(T_2-T_1,\dots,T_n-T_1)R$ is a canonical ideal of $R$ and
that $\omega_R^{[p]}=(T_1+\cdots + T_n)^{p-1}\omega_R$ (cf.
\cite[Example 2.8]{G}).  Hence, $\omega_R^{[p]}\cong \omega_R$ and
$R$ is weakly FPI by Theorem \ref{one-dim}.
\end{proof}

As a specific example of a one-dimensional non-Gorenstein FPI ring,
let $k$ be any field of characteristic $p$ and
$R=k[[x,y,z]]/(xy,xz,yz)$. Then $R$ is a one-dimensional local ring
which is $F$-pure  (and hence FPI) but not Gorenstein. Notice in
this example that $R$ has three associated primes. Regarding this we
note the following, which is a consequence of Corollary 1.3 of
\cite{G}:

\begin{cor} \label{Goto-cor}  Let $R$ be a one-dimensional complete local ring with algebraically closed residue field and suppose $R$ has at most two associated primes.  The following are equivalent:
\begin{enumerate}[(a)]
\item  $R$ is weakly FPI;
\item $R$ is Gorenstein.
\end{enumerate}
\end{cor}
\begin{proof}  Notice that the hypotheses imply that $R$ is $F$-finite.  Hence, (a) is equivalent to the condition that $\Hom_R(R^{f},R)\cong R$ by Corollary \ref{prop2}.
The result now follows from \cite[Corollary 1.3]{G}.
\end{proof}

\end{document}